\documentclass[12pt]{amsart}
\usepackage{graphicx}
\newtheorem{thm}{Theorem}[section]

\newtheorem{lem}[thm]{Lemma}

\theoremstyle{definition}

\theoremstyle{remark}
\newtheorem{rem}[thm]{Remark}
\numberwithin{equation}{section}
\begin{document}

\title[]{An extension of Van Vleck's functional equation for the sine}%
\author{Bouikhalene Belaid and Elqorachi Elhoucien}%
%\address{}%
%\email{}%

%\subjclass{}%

%\date{}%
%\dedicatory{}%
%\commby{}%
% ----------------------------------------------------------------
\begin{abstract}
In \cite{St3} H. Stetk\ae r obtained the solutions of Van Vleck's
functional equation for the sine $$f(x\tau(y)z_0)-f(xyz_0)
=2f(x)f(y),\; x,y\in G,$$ where $G$ is a semigroup, $\tau$ is an
involution of $G$ and $z_0$ is a fixed element in the center of $G$.
The purpose of  this paper is to determine the complex-valued
solutions of the following extension of Van Vleck's functional
equation for the sine $$\mu(y)f(x\tau(y)z_0)-f(xyz_0) =2f(x)f(y),
\;x,y\in G,$$ where $\mu$ : $G\longrightarrow \mathbb{C}$ is a
multiplicative function such that $\mu(x\tau(x))=1$ for all $x\in
G$. Furthermore, we obtain the solutions of a variant of Van Vleck's
functional equation for the sine
$$\mu(y)f(\sigma(y)xz_0)-f(xyz_0) = 2f(x)f(y),
\;x,y\in G$$ on monoids, and where $\sigma$ is an automorphism
involutive of $G$.
\end{abstract}
\maketitle
% ----------------------------------------------------------------
\section{Introduction}
In 1910, Van Vleck \cite{V1,V2} studied the continuous solution $f$
: $\mathbb{R} \longrightarrow \mathbb{R}$, $f\neq 0$ of the
following functional equation
\begin{equation}\label{eq1}
f(x - y + z_0)-f(x + y + z_0) = 2f(x)f(y),\;x,y \in
\mathbb{R},\end{equation} where $z_0>0$ is fixed. He showed first
that all solutions are periodic with period $4z_0$, and then he
selected for his study any continuous solution with minimal period
$4z_0$. He proved that such solution has to be the sine function
$f(x)=\sin(\frac{\pi}{2z_0}x)$, $x\in \mathbb{R}$. We refer also to \cite{K}.\\
In \cite{S}, Sahoo studied the following generalization
 \begin{equation}\label{eq2}
 f(x - y +z_0) + g(x + y + z_0) = 2f(x)f(y)\; x,y\in G\end{equation}
 of the functional equations (\ref{eq1}). He determined the general
solutions of this equation on an abelian group $G$. Stetk\ae r [8,
Exercise 9.18] found the complex-valued solution of equation
\begin{equation}\label{eq3} f(xy^{-1}z_0)-f(xyz_0) =
2f(x)f(y),\;x,y \in G,\end{equation}  when $G$ a group not
necessarily abelian and $z_0$ is a fixed element in the center of
$G$. Recently, Perkins and Sahoo \cite{P} replaced the group
inversion by the more general involution $\tau$: $G\longrightarrow
G$ and they obtained the abelian, complex-valued solutions of
equation
\begin{equation}\label{eq4} f(x\tau(y)z_0)-f(xyz_0) = 2f(x)f(y),\;x,y
\in G.\end{equation}  Stetk\ae r \cite{St3} extends the results of
Perkins and Sahoo \cite{P} about equation (\ref{eq4}) to the more
general case where $G$ is a semigroup and the solutions are not
assumed to be abelian.\\
The first purpose of this paper is to extend the results of Stetk\ae
r \cite{St3} to the following generalization of Van Vleck's
functional equation for the sine
\begin{equation}\label{eq5}
\mu(y)f(x\tau(y)z_0)-f(xyz_0) =2f(x)f(y), x,y\in G,\end{equation}
where $\mu$ is a  multiplicative function of a semigroup $G$ such
that $\mu(x\tau(x))=1$ for all $x\in G$.   As in the previous
results the main idea is to relate the functional equation
(\ref{eq4}) to the corresponding  d'Alembert's functional equation.
In our case we shall relate (\ref{eq5}) to the following version of
d'Alembert's functional equation
\begin{equation}\label{eq6}
    g(xy)+\mu(y)g(x\tau(y))=2g(x)g(y),\;x,y\in G
\end{equation} and we apply the crucial propositions [8, Proposition 9.17(c)] and [8
, Proposition
8.14(a)]. \\Replacing $f$ by $-F$ in (\ref{eq5}) we arrive at
$$F(xyz_0)-\mu(y)F(x\tau(y)z_0) = 2F(x)F(y),$$ which shows the similarity
between (\ref{eq5}) and (\ref{eq6}).\\
In Section 3, we obtain the solutions of a variant of Van Vleck's
functional equation for the sine
\begin{equation}\label{eq555}
\mu(y)f(\sigma(y)xz_0)-f(xyz_0) = 2f(x)f(y),\; x,y\in
G\end{equation}
 where $G$ is a monoid (semigroup with identity
element $e$), and  $\sigma$ is an automorphism involutive: That is
$\sigma(xy)=\sigma(x)\sigma(y)$ and $\sigma(\sigma(x))=x$ for all
$x,y$. In this case the main idea is to relate the functional
equation (\ref{eq555}) to  to the following variant of d'Alembert's
functional equation
\begin{equation}\label{eq666}
    g(xy)+\mu(y)g(\sigma(y)x)=2g(x)g(y),\;x,y\in G
\end{equation} and we apply the result obtained by Elqorachi and Redouani in [2, Lemma 3.2]. We refer also to \cite{stetkaer}
 in which the solutions of equation (\ref{eq666}) with $\mu=1$ are obtained on semigroups. \\The new feature of our paper is the introduction of the
multiplicative function $\mu$.
\section{Solutions of equation (\ref{eq5}) on semigroups}
Throughout this section $G$ denote a semigroup, $\tau$ :
$G\longrightarrow G$ is an involution of $G$. That is
$\tau(xy)=\tau(y)\tau(x)$ and $\tau(\tau(x))=x$ for all $x,y\in G$.
The element $z_0$ denotes a fixed element in the center of $G$.
Finally $\mu$ : $G\longrightarrow \mathbb{C}$ is a multiplicative
function such that $\mu(x\tau(x))=1$ for all $x\in G$. \\We first
prove the following lemmas which are generalizations of the useful
lemmas obtained by Stetk\ae r \cite{St3}.
\begin{lem} Let $f \neq 0$ be a solution of (\ref{eq5}). Then for
all $x\in G$ we have
\begin{equation}\label{eq7}
    f(x)=-\mu(x)f(\tau(x)),
\end{equation}
\begin{equation}\label{eq8}
    f(z_0)\neq 0,
\end{equation}
\begin{equation}\label{eq9'}
    f(z_{0}^{2})=0,
\end{equation}
\begin{equation}\label{eq9}
    f(x\tau(z_0)z_0)=\mu(\tau(z_0))f(x)f(z_0),
\end{equation}
\begin{equation}\label{eq10}
    f(xz_{0}^{2})=-f(z_0)f(x),
\end{equation}
\begin{equation}\label{eq11}
    \mu(x)f(\tau(x)z_0)=f(xz_0).
\end{equation}Let $G$ be a group and $\tau(x)=x^{-1}$ for all $x\in
G$. Then $f(z_0)=\mu(z_0)$ for any solution $f\neq 0$ of (\ref{eq5})
and  $f(xz_{0}^{4})=\mu(z_0)^{2}f(x)$ for all $x\in G.$
 \end{lem}
\begin{proof}If we replace $y$ by $\tau(y)$ in (\ref{eq5}) we obtain
\begin{equation}\label{eq12}
    \mu(\tau(y))f(xyz_0)-f(x\tau(y)z_0)=2f(x)f(\tau(y)).
\end{equation}
By multiplying (\ref{eq12}) by $\mu(y)$ and using the assumption
that  $\mu(y\tau(y))=1$ we get
\begin{equation}\label{eq13}
    \mu(y)f(x\tau(y)z_0)-f(xyz_0)=-2f(x)\mu(y)f(\tau(y)).
\end{equation}
By comparing (\ref{eq5}) with (\ref{eq13}) we get
$-f(x)\mu(y)f(\tau(y))=f(x)f(y)$ for all $x,y\in G$. Since $f\neq
0$, then we have (\ref{eq7}).\\
Setting $x=\tau(z_0)$ in (\ref{eq5}) and using (\ref{eq7}) we see
that
\begin{equation}\label{eq14}
    \mu(y)f(\tau(z_0)\tau(y)z_0)-f(\tau(z_0)yz_0)=2f(\tau(z_0))f(y)=-2\mu(\tau(z_0))f(z_0)f(y).
\end{equation}
On the other hand by using  (\ref{eq7}) and $\mu(x\tau(x))=1$ for
all $x\in G$ we get
$$\mu(y)f(\tau(z_0)\tau(y)z_0)=-\mu(y)\mu(\tau(z_0)\tau(y)z_0)f(\tau(z_0)yz_0)=-f(\tau(z_0)yz_0).$$
So, equation (\ref{eq14}) implies that
\begin{equation}\label{eq15}
f(\tau(z_0)yz_0)=\mu(\tau(z_0))f(z_0)f(y)
\end{equation}for all $y\in G$. Since $z_0,$ $\tau(z_0)$ are in
the center of $G$ then we obtain (\ref{eq9}).\\
Putting $y=z_0$ in  (\ref{eq5}) and using (\ref{eq9}) and
$\mu(z_0\tau(z_0))=1$ we get
$\mu(z_0)f(x\tau(z_0)z_0)-f(xz_{0}^{2})$
$=\mu(z_0)\mu(\tau(z_0))f(z_0)f(x)-f(xz_{0}^{2})=2f(x)f(z_0).$ So, we have $f(xz_{0}^{2})=-f(x)f(z_0)$ for all $x\in G.$ Which proves  (\ref{eq10}).\\
By replacing $x$ by $xz_0$ in the functional equation (\ref{eq5})
and using (\ref{eq10}) we obtain
$2f(xz_0)f(y)=\mu(y)f(x\tau(y)z{_0}^{2})-f(xyz_{0}^{2})=$
$-\mu(y)f(z_0)f(x\tau(y))+f(z_0)f(xy)$. \\If $f(z_0)=0$, then
$f(y)f(xz_0)=0$ for all $x,y\in G$. Since $f\neq 0$ then $f(xz_0)=0$
for all $x\in G$ so, we have $
\mu(y)f(x\tau(y)z_0)-f(xyz_0)=0=2f(x)f(y)$ for all $x,y\in G$ from
which we deduce that $f(x)=0$ for all $x\in G$. This contradicts the
assumption that $f\neq 0$ and it follows that
$f(z_0)\neq 0$.\\
By replacing $y$ by $z\tau(z)$ in (\ref{eq5}) and using
$\mu(z\tau(z))=1$ we get
$f(xz\tau(z)z_0)-f(xz\tau(z)z_0)=2f(x)f(z\tau(z))=0$ for all $x,z\in
G$. Since $f\neq 0$ then we get $f(z\tau(z))=0$ for all $z\in G.$
\\From (\ref{eq10}) and (\ref{eq7}) we have
$0=f(\tau(z_{0}^{2})z_{0}^{2})=-f(z_0)f(\tau(z_{0}^{2}))=\mu(\tau(z_{0}^{2}))f(z_{0}^{2})f(z_{0}).$
Since $f(z_{0})\neq 0$ we get
$$
\mu(\tau(z_{0}^{2}))f(z_{0}^{2})=0.$$ It follows from $1 =
\mu(x\tau(x)) = \mu(x)\mu(\tau(x))$, that $\mu(x)\neq 0$ for all
$x\in G.$ It is thus immediate  that
\begin{equation}\label{eq16}
f(z_{0}^{2} ) = 0.\end{equation}
 Putting $x=z_{0}^{2}$ in
(\ref{eq5}) and using (\ref{eq16}) to get
$\mu(y)f(z_{0}^{2}\tau(y)z_0)=f(z_{0}^{2}yz_0)$ and from
(\ref{eq10}) we obtain $-\mu(y)f(z_0)f(\tau(y)z_0)=-f(z_0)f(yz_0).$
Since $f(z_0)\neq 0$ then we get
\begin{equation}\label{eq17}
    \mu(y)f(\tau(y)z_0)=f(yz_0)
\end{equation}for all $y\in G.$ \\The statements
for the group case are consequences of the formulas (\ref{eq9}) and
(\ref{eq10}). This completes the proof.
\end{proof}
\begin{lem}Let $f\neq 0$ be a solution of equation (\ref{eq5}). Then\\
(a) the function defined by $$g(x)\;:=\frac{f(xz_0)}{f(z_0)}\;for
\;x\in G$$ is a non-zero abelian solution of d'Alembert's functional
equation (\ref{eq6}).\\(b) The function $g$ from (a) has the form
$g=\frac{\chi+\mu\chi\circ \tau}{2}$, where $\chi$ :
$G\longrightarrow \mathbb{C}$, $\chi\neq 0$, is a multiplicative
function.
\end{lem}
\begin{proof}By using (\ref{eq9}) and (\ref{eq10}) we get
$$f(z_0)^{2}[g(xy)+\mu(y)g(x\tau(y))]=\mu(y)f(z_0)f(x\tau(y)z_0)+f(z_0)f(xyz_0)$$
$$=\mu(y)\mu(z_0)f(x\tau(y)z_0\tau(z_0)z_0)-f(xyz_0z_{0}^{2})$$
$$=\mu(yz_0)f((xz_0)\tau(yz_0)z_0)-f((xz_0)(yz_0)z_0)=2f(xz_0)f(yz_0),$$ which implies that $g$ is a solution of equation (\ref{eq6}). Furthermore,
$g(z_{0}^{2})=f(z_{0}^{2}z_0)/f(z_0)=-f(z_0)f(z_0)/f(z_0)=-f(z_0)\neq
0$, so $g\neq 0.$\\As $g$ is a solution of equation (\ref{eq6}) then
by [8, Proposition 9.17(c)] $g$ is a solution of pre-d'Alembert
function. Note that $g(z_0)=f(z_{0}^{2})/f(z_0)=0$,
$d(z_0)=2g(z_0)^{2}-g(z_{0}^{2})=0-(-f(z_0))=f(z_0)\neq 0$. So, we
have $g(z_0)^{2}\neq d(z_0)$ and according to [8, Proposition
8.14(a)] $g$ is  abelian. By using the result obtained by Stetk\ae r
in  [8, Proposition 9.31, p. 158] we get the proof of (b). We refer
also to \cite{St1}.
\end{proof}
Now, we are ready to prove the first main result of the present
paper.
\begin{thm} The non-zero solution $f$ : $G\longrightarrow \mathbb{C}$ of
the functional equation (\ref{eq5}) are the functions of the form
\begin{equation}\label{eq300}
    f=\mu(z_0)\chi(\tau(z_0))\frac{\chi-\mu\chi\circ\tau}{2}=\chi(z_0)\frac{\mu\chi\circ\tau-\chi}{2},
\end{equation}where $\chi$ : $G\longrightarrow \mathbb{C}$ is a
multiplicative function such that $\chi(z_0)\neq 0$ and
$\mu(z_0)\chi(\tau(z_0))=-\chi(z_0)$. Furthermore,
$f(z_0)=\mu(z_0)\chi(z_0\tau(z_0))$.\\If $G$ is a group and $\tau$
is the group inversion, then $\chi(z_{0}^{2})=-\mu(z_0)$ and
$f(z_0)=\mu(z_0)$ for any non-zero solution of equation
(\ref{eq5}).\\If $G$ is a topological semigroup and that $\tau$ :
$G\longrightarrow G$, $\mu$ : $G\longrightarrow \mathbb{C}$ are
continuous, then the non-zero solution $f$ of equation (\ref{eq5})
is continuous, if and only if $\chi$ is continuous.
\end{thm}\begin{proof}Let $f$ : $G\longrightarrow \mathbb{C}$ be a non-zero solution of equation (\ref{eq5}). Then we get
$$f(x)=\frac{\mu(z_0)f(x\tau(z_0)z_0)-f(xz_0z_0)}{2f(z_0)}=\frac{\mu(z_0)g(x\tau(z_0))-g(xz_0)}{2}$$ for all $x\in G$
and where $g$ is the function defined in Lemma 2.2. So, from Lemma
2.2 we have $g=\frac{\chi+\mu\chi\circ\tau}{2}$ and then after an
easy computation we obtain
\begin{equation}\label{eq301}
    f=\frac{\chi(z_0)-\mu(z_0)\chi(\tau(z_0))}{2}\frac{\mu\chi\circ\tau-\chi}{2}.
\end{equation} From (\ref{eq11}) we have
\begin{equation}\label{eq302}
\mu(x)f(\tau(x)z_0)=f(xz_0)\end{equation}
 for all $x\in G$.
Substituting  (\ref{eq301})  into (\ref{eq302}) we get after an
elementary computation that
$$[\mu(z_0)\chi(\tau(z_0))+\chi(z_0)][\chi-\mu\chi\circ\tau]=0.$$ The rest of
the proof is similar to the one  by Stetk\ae r \cite{St3}.\\ Under
the hypotheses of Theorem 2.3 any solution $f$ of (\ref{eq5}) is
abelian.
\end{proof}
\section{Solutions of equation (\ref{eq555}) on monoids}
Throughout this section $G$ denote a monoid, $\sigma$ :
$G\longrightarrow G$ is an automorphism involutive of $G$. That is
$\sigma(xy)=\sigma(x)\sigma(y)$ and $\sigma(\sigma(x))=x$ for all
$x,y\in G$. The element $z_0$ denotes a fixed element in the center
of $G$. Finally $\mu$ : $G\longrightarrow \mathbb{C}$ is a
multiplicative function such that $\mu(x\sigma(x))=1$ for all $x\in
G$.\\The following useful lemma will be used later.
\begin{lem} Let $f \neq 0$ be a solution of (\ref{eq555}). Then \\ (a) For
all $x\in G$ we have
\begin{equation}\label{eqo7}
    f(x)=-\mu(x)f(\sigma(x)),
\end{equation}
\begin{equation}\label{eqo8}
    f(z_0)\neq 0,
\end{equation}
\begin{equation}\label{eqo9}
    f(x\sigma(z_0)z_0)=\mu(\sigma(z_0))f(x)f(z_0),
\end{equation}
\begin{equation}\label{eqo10}
    f(xz_{0}^{2})=-f(z_0)f(x),
\end{equation}
\begin{equation}\label{eqo11}
    \mu(x)f(\sigma(x)z_0)=f(xz_0).
\end{equation} (b) the function defined by $$g(x)\;:=\frac{f(xz_0)}{f(z_0)}\;for
\;x\in G$$ is a non-zero  solution of a variant of d'Alembert's
functional equation
\begin{equation}\label{EQUATIONES}
    g(xy)+\mu(y)g(\sigma(y)x)=2g(x)g(y),\;x,y\in G.
\end{equation}
(c) The function $g$ from (b) has the form
$g=\frac{\chi+\mu\chi\circ \sigma}{2}$, where $\chi$ :
$G\longrightarrow \mathbb{C}$, $\chi\neq 0$, is a multiplicative
function.
 \end{lem}\begin{proof} (a) Putting $y=e$ in (\ref{eq555}) we get $f(xz_0)-f(xz_0)=2f(x)f(e)=0$. Since $f\neq0$ then we have $f(e)=0$.\\
 Taking $x=e$ in (\ref{eq555}) and using $f(e)=0$ we get
 (\ref{eqo11}). \\By replacing $x$ by $\sigma(x)$ in (\ref{eq555}) we
 obtain $$\mu(y)f(\sigma(y)\sigma(x)z_0)-f(\sigma(x)yz_0)
 =2f(\sigma(x))f(y).$$ Since from (\ref{eqo11}) we have $$\mu(y)f(\sigma(y)\sigma(x)z_0)=\mu(y)f(\sigma(yx)z_0)$$
 $$=\mu(y)\mu(\sigma(yx))f(yxz_0)=\mu(\sigma(x))f(yxz_0)$$ and it follows that
 $$\mu(x)[\mu(\sigma(x))f(yxz_0)-f(\sigma(x)yz_0)]=\mu(x)2f(\sigma(x))f(y)=f(yxz_0)-\mu(x)f(\sigma(x)yz_0).$$
 Since
 $$f(yxz_0)-\mu(x)f(\sigma(x)yz_0)=-[\mu(x)f(\sigma(x)yz_0)-f(yxz_0)]=-2f(y)f(x).$$
 So, we conclude that $$-2f(x)f(y)=2\mu(x)f(\sigma(x))f(y)$$ for all
 $x,y\in G$. Since $f\neq0$ then we get (\ref{eqo7}).\\
 Putting $x=\sigma(z_0)$ in (\ref{eq555}) and using (\ref{eqo7}) we get $$\mu(y)f(\sigma(y)\sigma(z_0)z_0)-f(\sigma(z_0)yz_0) =2f(\sigma(z_0))f(y)=
 -2\mu(\sigma(z_0))f(z_0)f(y).$$ On the other hand we have
 $$\mu(y)f(\sigma(y)\sigma(z_0)z_0)=-\mu(y)\mu(\sigma(y)\sigma(z_0)z_0)f(yz_0\sigma(z_0))=-f(yz_0\sigma(z_0)).$$
 Then, since $z_0$, $\sigma(z_0)$ are in the center of $G$ we get (\ref{eqo9}).\\
 By replacing $y$ by $z_0$ in (\ref{eq555}) and using (\ref{eqo9}) we
 get $$\mu(z_0)f(\sigma(z_0)xz_0)-f(xz_{0}^{2})
 =2f(x)f(z_0)=\mu(z_0)\mu(\sigma(z_0))f(x)f(z_0)-f(xz_{0}^{2})=f(x)f(z_0)-f(xz_{0}^{2}).$$
 So, we deduce equation (\ref{eqo10}).\\
 By replacing  $x$ by $xz_0$ in  (\ref{eq555}) and using (\ref{eqo10}) we get $$\mu(y)f(\sigma(y)xz_{0}^{2})-f(xz_0yz_0)=2f(xz_0)f(y)$$
 $$=-\mu(y)f(\sigma(y)x)f(z_0)+f(xy)f(z_0).$$If $f(z_0)=0$, then
 $f(xz_0)f(y)=0$ for all $x,y\in G$, since $f\neq0$ then $f(xz_0)=0$
 for all $x\in G$ and from (\ref{eq555}) we obtain $2f(x)f(y)=0$
 for all $x,y\in G$. This contradict the assumption that $f\neq0$ and this proves (\ref{eqo8}).\\
 (b) For all $x,y\in G$ we have $$f(z_0)^{2}[g(xy)+\mu(y)g(\sigma(y)x)]=f(z_0)\mu(y)f(\sigma(y)xz_0)+f(z_0)f(xyz_0)$$
 $$=\mu(z_0)\mu(y)f(\sigma(y)xz_0\sigma(z_0)z_0)-f(xyz_0z_{0}^{2})=\mu(yz_0)f(\sigma(yz_0)xz_0z_0)-f((xz_0)(yz_0)z_0)$$$$=2f(xz_0)f(yz_0).$$
 Dividing by $f(z_0)^{2}$ we get (b). Furthermore, $g(e)=1$ then $g\neq
 0.$\\ Now, from [2, Lemma 3.2] we get (c) and this
 completes the proof.\\The secand main result of this paper  is:
 \end{proof}
 \begin{thm} The non-zero solution $f$ : $G\longrightarrow \mathbb{C}$ of
the functional equation (\ref{eq555}) are the functions of the form
\begin{equation}\label{eq300}
    f=f=\mu(z_0)\chi(\sigma(z_0))\frac{\chi-\mu\chi\circ\sigma}{2}=\chi(z_0)\frac{\mu\chi\circ\sigma-\chi}{2},
\end{equation}where $\chi$ : $S\longrightarrow \mathbb{C}$ is a
multiplicative function such that $\chi(z_0)\neq 0$ and
$\mu(z_0)\chi(\sigma(z_0))=-\chi(z_0)$. Furthermore,
$f(z_0)=\mu(z_0)\chi(z_0\sigma(z_0))$.\\If $G$ is a topological
monoid and that $\sigma$ : $G\longrightarrow G$, $\mu$ :
$G\longrightarrow \mathbb{C}$ are continuous, then the non-zero
solution $f$ of equation (\ref{eq555}) is continuous, if and only if
$\chi$ is continuous.
\end{thm}\begin{proof} Let $f$ : $G\longrightarrow \mathbb{C}$ be a
non-zero solution of the functional equation (\ref{eq555}). Putting
$y=z_0$ in (\ref{eq555}) we get
\begin{equation}\label{elqorachi12}
f(x)=\frac{\mu(z_0)f(\sigma(z_0)xz_0)-f(xz_0z_0)}{2f(z_0)}=\frac{1}{2}(\mu(z_0)g(\sigma(z_0)x)-g(xz_0))\end{equation}
for all $x\in G$ and where g is the function defined in Lemma 6.1.
So, from Lemma 6.2 we have $g=\frac{\chi+\mu\chi\circ \sigma}{2}$,
where $\chi$ : $G\longrightarrow \mathbb{C}$, $\chi\neq 0$, is a
multiplicative function. Substituting this into (\ref{elqorachi12})
we find that $f$ has the form
\begin{equation}\label{equo1àà}
    f=\frac{\chi(z_0)-\mu(z_0)\chi(\sigma(z_0))}{2}\frac{\mu\chi\circ\sigma-\chi}{2}.
\end{equation} We have from (\ref{eqo11}) that $f$ satisfies
\begin{equation}\label{equa2àà}
\mu(x)f(\sigma(x)z_0)=f(xz_0)\end{equation}
 for all $x\in G$.
Applying the last expression of $f$ in  (\ref{equa2àà}) gives after
an elementary computations that
$$[\mu(z_0)\chi(\tau(z_0))+\chi(z_0)][\chi-\mu\chi\circ\tau]=0.$$ Since $\chi\neq \mu\chi\circ\sigma$, we deduce that $\mu(z_0)\chi(\tau(z_0))+\chi(z_0)=0$.
So, (\ref{equo1àà}) can be written as follows
$$f=-\mu(z_0)\chi(\sigma(z_0))\frac{\chi-\mu\chi\circ\sigma}{2}=\chi(z_0)\frac{\chi-\mu\chi\circ\sigma}{2}.$$
For the topological statement we use [8, Theorem 3.18(d)]. This ends
the proof. \end{proof} \begin{rem}On groups the solutions of the
functional equation
$$\mu(y)f(\sigma(y)xz_0)+g(xyz_0)=h(x)h(y), \; x,y\in G,$$ where $z_0$ is an arbitrary element of $G$ (not necessarily in the center of $G$) can be found by putting
$f_1(x)=f(xz_0)$; $f_2(x)=g(xz_0)$, so we have $f_1,f_2$ satisfy the
functional equation
$$\mu(y)f_1(\sigma(y)x)+f_2(xy)=h(x)h(y),\; x,y\in G$$ which was
solved by the authors in \cite{preprint}.
\end{rem}

% ----------------------------------------------------------------

Author's address:\\\\  Bouikhalene Belaid, Polydisciplinary Faculty,
University Sultan Moulay Slimane, Beni-Mellal, Morocco, E-mail :
bbouikhalene@yahoo.fr\\ \\ Elqorachi Elhoucien, Department of
Mathematics, Faculty of Sciences, University Ibn Zohr, Agadir,
Morocco, E-mail: elqorachi@hotamail.com
\end{document}